\documentclass[12pt]{article}
\usepackage[utf8]{inputenc}
\usepackage{amsmath}
\usepackage{amssymb}
\usepackage{amsthm}
\usepackage{changepage}
\usepackage{ dsfont }
\usepackage{soul}
\usepackage{hyperref}
\usepackage[dvipsnames]{xcolor}

\hypersetup{
    colorlinks=true,
    linkcolor=blue,
    filecolor=magenta,      
    urlcolor=cyan,
    citecolor=Green,
    }

\usepackage{tikz}

\usepackage{graphicx}

\theoremstyle{definition}
\newtheorem{theorem}{Theorem}[section]
\newtheorem{lemma}[theorem]{Lemma}
\newtheorem{proposition}[theorem]{Proposition}

\newtheorem{conjecture}{Conjecture}

\newtheorem{corollary}[theorem]{Corollary}
\newtheorem{definition}[theorem]{Definition}

\title{Monochromatic products and sums in $2$-colorings of $\mathbb{N}$}

\author{Matt Bowen\thanks{McGill University, Montréal, Quebec, Canada. \texttt{matthew.bowen2@mail.mcgill.ca} }}

\date{May 2022}

\begin{document}

\maketitle

\begin{abstract}
    We show that any $2$-coloring of $\mathbb{N}$ contains infinitely many monochromatic sets of the form $\{x,y,xy,x+y\},$ and more generally monochromatic sets of the form $\{x_i,\prod x_i,\sum x_i: i\leq k\}$ for any $k\in\mathbb{N}.$  Along the way we prove a monochromatic products of sums theorem that extends Hindman's theorem and a colorful variant of this result that holds in any 'balanced' coloring.
\end{abstract}

\section{Introduction}

The following conjecture of Hindman is well known, and several variants and relaxations of it have been the subject of recent attention in \cite{BM1, BM2, GG, GS, Mor, P}.

\begin{conjecture}\label{conjecture}
Any finite coloring of $\mathbb{N}$ contains monochromatic sets of the form $\{x,y,xy,x+y\}.$
\end{conjecture}

This question was first studied in 1977 by Graham, who showed that any $2$-coloring of $\{1,...,252\}$ contains a monochromatic configuration of this form, and by Hindman, who showed that any $2$-coloring of $\{2,...,990\}$ contains such a configuration \cite{Hc}.  However, these results were proven using brute force computation and left open the problem of determining if infinitely many monochromatic sets of this type can be found in any $2$-coloring of $\mathbb{N}$.

Our main result solves a generalization of this problem.

\begin{theorem}\label{main}
Given any $2$-coloring of $\mathbb{N}$ and $n\in\mathbb{N}$ there are arbitrarily large and distinct $x_1,...,x_n\in\mathbb{N}$ such that $$\{x_i,\prod_{j\leq i}x_j, \sum_{j=1}^n x_j: i\leq n\}$$ is monochromatic. 
\end{theorem}

Our methods can also be used to prove other results of a similar flavor.  For example, in \cite{GG} Goldoni and Goldoni showed that any $2$-coloring of $\{1,...,44\}$ contains (not necessarily distinct) $x,y$ such that $\{x,y,xy,x+2y\}$ is monochromatic.  We prove the following.

\begin{theorem}\label{m2}
Given any $2$-coloring of $\mathbb{N}$ and $n\in\mathbb{N}$ there are monochromatic sets of the form $\{x,y,xy,x+ny\}.$ 
\end{theorem}

\subsection{Techniques}

The approach in this paper builds off of recent work of Moreira, who proved the following result in \cite{Mor}.

\begin{theorem}[Moreira's Theorem]\label{M}
For any $n,k\in\mathbb{N}$ and finite coloring of $\mathbb{N},$ there is a monochromatic set of the form $$\{\prod_{j\leq i}x_j+\sum_{i<j\leq n}c_jx_j: i\leq n, \textnormal{ } 0\leq c_j\leq k\}.$$
\end{theorem}

Notice that Theorem \ref{M} would imply Theorems \ref{main} and \ref{m2} if we could ensure that the 'step sizes' $\{x_2,...,x_n\}$ were also the right color.  Our main idea is that this is possible to arrange in any sufficiently 'balanced' coloring of the naturals. To achieve this, we will work with the following generalization of Hindman's theorem \cite{Hind}.

\begin{theorem}\label{pos}
For any finite coloring of $\mathbb{N}$ and any function $f:\mathbb{N}\rightarrow\mathbb{N}$ there is a color class $C$ and sequence of sets $S_1,S_2...\subset \mathbb{N}$ with $|S_i|=f(i)$ such that for any finite $F\subset \bigcup_i S_i$ we have $$\prod_{i: S_i\cap F\neq \emptyset}\sum_{s\in S_i\cap F} s\in C.$$
\end{theorem}

When we apply the above theorem to the constant function $f(n)=1$ we recover Hindman's finite products theorem, and when we apply it to the function $f(1)=n$ and $f(i)=0$ otherwise we recover the finite sums theorem of Folkman-Rado-Sanders.  As we will see, the step sizes $x_i$ in Moreira's Theorem \ref{M} are closely related to the other configurations that we can obtain from Theorem \ref{pos}, which we will leverage to prove our main results in the case that the colorings are sufficiently 'balanced.'

More precisely, recall that $S\subseteq \mathbb{N}$ is \textbf{multiplicatively syndetic} if there is a finite set $F\subset \mathbb{N}$ such that $\mathbb{N}=S/F,$ and that $T\subseteq \mathbb{N}$ is \textbf{multiplicatively thick} if its compliment is not syndetic, or equivalently if for each finite $F\subset \mathbb{N}$ there is a $a\in T$ with $aF\subset T$.  We will call a $2$-coloring $\mathbb{N}=C_0\cup C_1$ \textbf{locally balanced} if there is a thick set $T$ and syndetic sets $S_0,S_1$ with $S_i\cap T\subseteq C_i.$  In this setting we may prove a colorful extension of Theorem \ref{pos}.

\begin{proposition}\label{balanced pos}
Suppose that $\mathbb{N}=C_0\cup C_1$ is a locally balanced coloring.  There is an $i\in\{0,1\}$  such that for any $f:\mathbb{N}\rightarrow\mathbb{N}$ there is an arbitrarily large finite $S_0\subset \mathbb{N}$ and a sequence of sets $S_1,S_2...\subset \mathbb{N}$ with $|S_j|=f(j)$ such that for any finite $F\subset \bigcup_j S_j$ we have $$\prod_{j: S_j\cap F\neq \emptyset}\sum_{s\in S_j\cap F} s $$

is in $C_i$ if $S_0\cap F\neq \emptyset$ and is in $C_{1-i}$ otherwise.

\end{proposition}

This proposition was inspired by other colorful Ramsey theorems for graphs and tournaments where each color class is sufficiently well represented; see \cite{BHMM,BLM,CHM,CM, GN,FS} for several recent results of a similar flavor.  By using Proposition \ref{balanced pos} together with the proof of Theorem \ref{M} we obtain the following:

\begin{theorem}\label{thm:local}
Any locally balanced $2$-coloring of $\mathbb{N}$ contains a monochromatic set of the form $$\{x_i,\prod_{j\leq i}x_i+\sum_{j<i\leq n}c_jx_j: i\leq n, \textnormal{ } 0\leq c_j\leq k\}.$$
\end{theorem}

From here we must only deal with the case when one color class is thick, which can be handled by adapting the proof of the two-color Schur theorem.

\section{Background information}

Our results rely on the algebraic structure of the semigroups $(\beta\mathbb{N},+)$ and $(\beta\mathbb{N},\cdot)$ and some combinatorial notions of size derived from this structure.  In this section we briefly review the results from this theory that we will use.  See \cite{HS} and \cite{BG} for a more thorough treatment.

First, we recall that the set of ultrafilters on $\mathbb{N}$, $\beta\mathbb{N},$ may be equipped with the topology generated by the clopen sets $$\overline{A}=\{p\in\mathbb{N}: A\in p\} \textnormal{ } A\subseteq \mathbb{N}.$$  

Moreover, we can extend a semigroup operation $\circ:\mathbb{N}^2\rightarrow \mathbb{N}$ to $\beta\mathbb{N}$ by defining for $p,q\in\beta\mathbb{N}$ $$p\circ q=\{A\subseteq\mathbb{N}: \{s: \{t: s\circ t\in A\}\in q\}\in p\}.$$

We will denote by $s^{-1}\circ A:=\{t: s\circ t\in A\}.$  In this language $A\in p\circ q$ if and only if $\{s: s^{-1}\circ A\in q\}\in p.$  

In this way $(\beta\mathbb{N},\circ)$ can be seen to be a compact right topological semigroup, and so by the Ellis-Numakura theorem (\cite{HS}, 2.6) we know that any left ideal in $(\beta\mathbb{N},\circ)$ contains a minimal left ideal, and in turn any minimal left ideal contains an idempotent element.  Moreover, as a consequence of \cite{HS} 1.58 we also know that $L=\beta\mathbb{N}\circ e$ for any minimal left ideal $L$ and idempotent $e\in L$.  We will call such an idempotent ultrafilter \textit{minimal}.  We summarize these facts in the following lemma.

\begin{lemma}\label{mi}
Any left ideal  $L\subseteq (\beta\mathbb{N},\cdot)$ contains a minimal left ideal $L$ which contains an idempotent element.  Moreover, if $e\in L$ is idempotent, then $p\cdot e=p$ for any $p\in L$.  
\end{lemma}

There are several combinatorial notions of size related to minimal left ideals in $\beta\mathbb{N}$ that we will need to make use of.

\begin{definition}
Let $(\mathbb{N},\circ)$ be a semigroup.
\begin{enumerate}
    \item We say $T\subseteq \mathbb{N}$ is $\circ$-thick if for any finite $F\subset \mathbb{N}$ there is an $a\in T$ such that $F\circ a\subset T$.  Equivalently (\cite{HS} 4.48), if there is a minimal left ideal $L\subset (\beta\mathbb{N},\circ)$ with $L\subseteq \overline{T}.$ 
    \item We say $S\subseteq \mathbb{N}$ is $\circ$-syndetic if there is a finite set $F\subset \mathbb{N}$ such that $\mathbb{N}=\bigcup_{s\in F}Ss^{-1}=S/F$.  Equivalently (\cite{HS} 4.4) if $\overline{S}\cap L\neq \emptyset$ for every minimal left ideal $L$.
    \item We say $A$ is $\circ$-piecewise syndetic if $A$ is the intersection of a thick set and a piecewise syndetic set.  Equavalently, if $\overline{A}\cap L\neq \emptyset$ for some minimal left ideal.
    \item We say $A$ is $\circ$-central if there is a minimal idempotent $p$ with $A\in p.$
    
    If we do not include a prefix then we mean the semigroup operation given by the usual $+$ operation on $\mathbb{N}.$
\end{enumerate}

\end{definition}

We will use the following facts about additively and multiplicatively piecewise syndetic sets, whose proofs can be derived from \cite{HS} Theorems 4.40, 5.19.2., and 5.8 along with the dilation invariance of finite sums.  

\begin{lemma}\label{lem:ps}

\

\begin{enumerate}

    \item If $A$ is peicewise syndetic and $A=A_1\cup A_2\cup...\cup A_r$, then some $A_i$ is piecewise syndetic.
    \item For any $k\in\mathbb{N}$, $A$ is piecewise syndetic if and only if $kA$ is piecewise syndetic.
    \item Given a sequence $(x_i)_{i\in\omega}$ let $FS(x_i)$ be the set containing all the non-repeating finite sums of $x_i$. If $C$ is multiplicatively piecewise syndetic then there are arbitrarily long sequences $(x_i)_{i\leq n}$ such that $FS(x_i)\subset C.$
\end{enumerate}

\end{lemma}

Finally, we will make use of the following variant of van der Waerden's theorem, which follows from the standard deduction of van der Waerden's theorem from the Hales-Jewett theorem.

\begin{theorem}[IP$^*_r$  van der Waerden]\label{thm:poly}

Let $A$ be a piecewise syndetic set and $k\in \mathbb{N}$.  Then there is an $n\in\mathbb{N}$ such that for any $S\subset\mathbb{N}$ with $|S|>n$ there is a $d\in FS(S)$ such that $$\bigcap_{0\leq i\leq k}(A-id)$$ is piecewise syndetic.
\end{theorem}

\section{Proofs of the main results}

\subsection{A colorful Hindman theorem}

We begin by proving refinements of Theorems \ref{pos} and \ref{balanced pos}.

\begin{theorem}\label{pos+}
For any multiplicatively central set $A$ and any function $f:\mathbb{N}\rightarrow\mathbb{N}$ there is a sequence of sets $S_1,S_2,...\subset \mathbb{N}$ with $|S_i|=f(i)$ such that for any finite $F\subset \bigcup_i S_i$ we have $$\prod_{i: S_i\cap F\neq \emptyset}\sum_{s\in S_i\cap F} s\in A.$$
\end{theorem}

\begin{proof}
This follows from a simple modification of the Galvin-Glazer proof of Hindman's Theorem.

Fix a function $f:\mathbb{N}\rightarrow\mathbb{N}$. 

As $A$ is multiplicatively central there is a minimal idempotent ultrafilter $e\in (\beta\mathbb{N},\cdot)$ with $A\in e.$  We inductively build the desired sets $S_i$ as follows.  First, note that $A^*=\{n: n^{-1}A\in e\}\cap A\in e$, and so by Lemma \ref{lem:ps} (3) we may find a finite set $S_1$ with $|S_1|=f(1)$ such that $FS(S_1)\subset A^*.$  Therefore, we may let $A_1=\bigcap_{s\in FS(S_1)}s^{-1}A\cap A^*$ and note that this is an element of $e$.  Let $A_1^*=\{n: n^{-1}A_1\in e\}\cap A.$  

Inductively assume we have defined $S_i$, $A_i$, and $A_i^*$ as above.  We define $S_{i+1},$ $A_{i+1}$, and $A_{i+1}^*$ in essentially the same way as before.  First, by Lemma \ref{lem:ps} (3) there is a set $S_{i+1}$ with $|S_{i+1}|=f(i+1)$ such that $FS(S_{i+1})\in A_i^*.$ Once we have defined $S_{i+1}$ in this way, we let $A_{i+1}=\bigcap_{s\in FS(S_{i+1})}s^{-1}A_i^*\cap A_{i}^*$.  

To verify that these choices of $S_i$ work, observe that $FS(S_i)\subset A$ for any $i$.  Now, let $s_i\in FS(S_i)$.  As $s_i\in \bigcap_{s\in FS(S_{j+1})}s^{-1}A_j^*\cap A_j^*$ for all $j+1<i$, we see that the full claim holds.
\end{proof}

\begin{theorem}\label{pos++}
Let $L\subset (\beta\mathbb{N},\cdot)$ be a minimal left ideal, $e\in L$ an idempotent, and $p\in L.$  Then for any $A\in p,$ $B\in e,$ and $f:\omega\rightarrow\omega$ there is a sequence of sets $S_0,S_1,...\subset \mathbb{N}$ with $|S_i|=f(i)$ such that for any finite $F\subset \bigcup_i S_i$ we have $$\prod_{i: S_i\cap F\neq \emptyset}\sum_{s\in S_i\cap F} s$$ is in $A$ if $S_0\cap F\neq \emptyset$ and is in $B$ otherwise.
\end{theorem}

\begin{proof}
By Lemma \ref{mi} we know that $pe=p.$  Therefore, by Lemma \ref{lem:ps} (3) we may pick $S_0$ such that $FS(S_0)\subset \{n:n^{-1}A\in e\}\cap A.$  Finally, we may apply Theorem \ref{pos+} to $B^*=\bigcap_{s\in FS(S_0)}s^{-1}A\cap B\in e$ to complete the proof.
\end{proof}

As already noted, these theorems contain both the finite products and finite sums theorems as special cases.  We also obtain the following slight refinement of a result of Hindman originally proven in \cite{Hsep}.

\begin{corollary}\label{central pos}
For any multiplicatively central set $A$ and $m,n\in\mathbb{N}$ there are sets $F,G\subset\mathbb{N}$ such that $|F|=m,$ $|G|=n,$ $FS(F)\cup FP(G)\subset A$ and $\sum F=\prod G.$

\end{corollary}

\begin{proof}
Find $S_0,...,S_{m}$ from Theorem \ref{pos+} with $|S_i|=n$.  Let $T=\{s_1,...,s_{m}\}$ with $s_i\in S_i$ arbitrary, $a=\prod T$, and $b=\sum S_0$.  Observe that $G=\{b,s_1,s_2,...,s_{m}\}$ and $F=aS_0$ satisfy the desired properties. 
\end{proof}

\subsection{The balanced case}

In this subsection we prove our main results in the special case of balanced colorings.  We begin with a separate proof of the $n=2$ and $k=1$ case, as this case contains all of the ideas needed for the general result while requiring much less notation.

\begin{figure}
\centering
    \includegraphics[scale=.6]{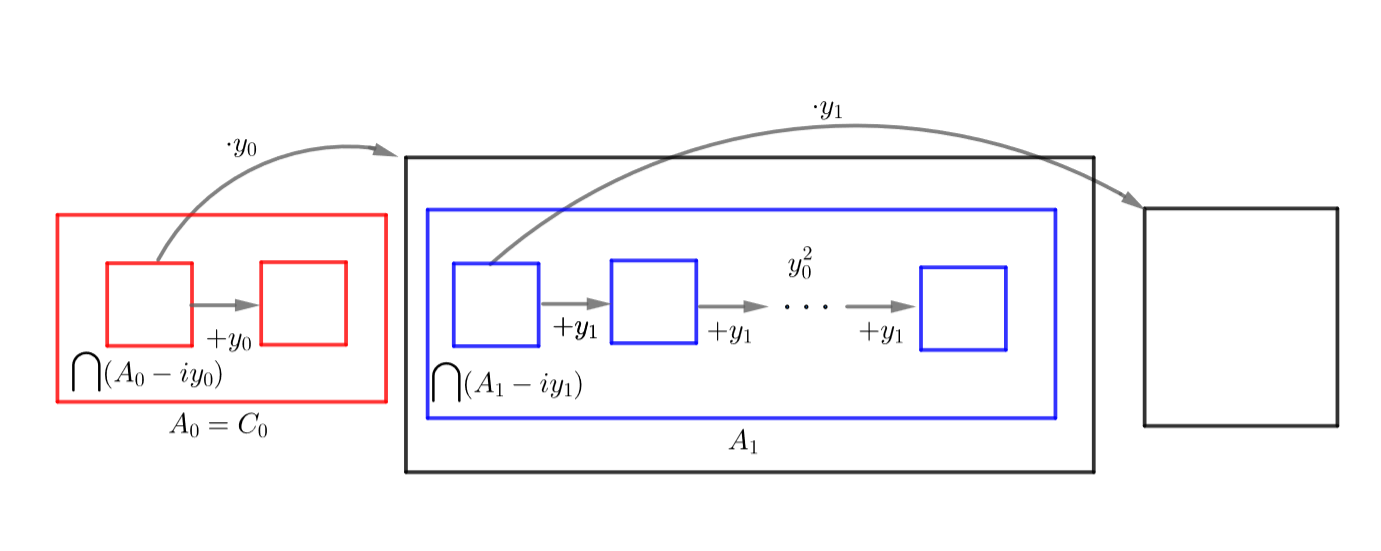}
    \caption{The proof of Theorem \ref{thm:local} when $n=2$ and $k=1.$  Each box represents a piecewise syndetic set.  We use Theorem \ref{pos++} along with Theorem \ref{thm:poly} to ensure that $y_0$ is red, $y_1$ is blue, and $y_0y_1$ is red.  If the final box contains a blue point $x',$ then $\{x'/y_1, y_1,x',x'/y_1+y_1\}$ is blue.  If this box contains a red point $x'$, then $\{x'/y_0y_1, y_0y_1, x', (x'/y_1+y_0^2y_1)/y_0\}$ is red.}
    \label{fig:new}
\end{figure}

\begin{theorem}
Let $T$ be a multiplicatively thick set, $M_0,M_1$ multiplicatively syndetic sets, and $\mathbb{N}=C_0\cup C_1$ a coloring such that $T\cap M_i\subseteq C_i.$  There is a monochromatic set of the form $\{x,y,xy,x+y\}.$
\end{theorem}

\begin{proof}
Without loss of generality we may assume that there is a minimal left ideal $L\subset (\beta\mathbb{N},\cdot)$ with $C_0\in p\in L$ and $C_1\in e\in L,$ with $e$ idempotent.  We may also assume that both $C_i$ are additively piecewise syndetic, as otherwise we may apply Lemma \ref{thm:poly} and Lemma \ref{lem:ps} (3) to find a piecewise syndetic set of starting points of monochromatic progressions whose step sizes are the same color.  As multiplying this set by the step size remains piecewise syndetic, we would have the desired result.  

By applying Theorem \ref{pos++} we may find a sequence of as large as desired sets $S_0,S_1,...\subset \mathbb{N}$  with $FS(S_0)\subseteq C_0,$ $FS(S_i)\subseteq C_1,$ and $FS(S_0)\cdot FS(S_i)\subseteq C_0$ for all $i>0.$  Applying Theorem \ref{thm:poly} with $C_0$ and $S_0$ we find a $y_0\in C_0$ such that $A_0'=C_0\cap (C_0-y_0)$ is piecewise syndetic.  Notice that if there is some $x'\in y_0\cdot A_0'\cap C_0$ then we'd have $$\{\frac{x'}{y_0},y_0,\frac{x'}{y_0}+y_0,x'\}\subset C_0,$$ and so we may assume that $A_1:= y_0\cdot A_0'\cap C_1$ is piecewise syndetic.  Applying \ref{thm:poly} to $A_1$ and a sufficiently large $S_i,$ we find a $y_1\in C_1$ such that $A_1'=A_1\cap(A_1-y_1)\cap (A_1-y_0^2y_1)$ is piecewise syndetic.  If $y_1\cdot A_1'\cap C_1\neq \emptyset$ then we may finish the proof as before, so we may assume that $y_1\cdot A_1'\subseteq C_0,$ and thus contains an element $x'.$  Note that $y=y_0y_1\in C_0$ and similarly $x=\frac{x'}{y}\in C_0.$  Finally, we have that $$x+y=(x'/y_1+y_0^2y_1)/y_0\in A_1/y_0\subseteq C_0$$ and so $\{x,y,xy,x+y\}\in C_0$ is the desired set.
\end{proof}

Now we move on to the proof of our more general result.  As already mentioned, the idea is essentially the same as in the above, but now we must be more careful in setting up our induction.

\vspace{2mm}
\begin{proof}[Proof of Theorem \ref{thm:local}]
Without loss of generality we may assume that there is a minimal left ideal $L\subset (\beta\mathbb{N},\cdot)$ with $C_0\in p\in L$ and $C_1\in e\in L,$ with $e$ idempotent.  We may also assume that both $C_i$ are additively piecewise syndetic, as otherwise we may apply Lemma \ref{thm:poly} and Lemma \ref{lem:ps} (3) to find a piecewise syndetic set of starting points of monochromatic progressions whose step sizes are the same coloring.  As multiplying this set by the step size remains piecewise syndetic, we would have the desired result.  

Our goal is to find a sequence of piecewise syndetic sets $A_i,$ and sequences of naturals $t_i$ and $y_i$ such that for each $i\in\mathbb{N}$:

\begin{enumerate}
    \item $$A_{m+1}=C_{t_{i+1}}\cap y_m(A_{m}\cap\bigcap_{f=1}^{\prod_{j<m} k!y_j^3} (A_{m}-fy_m))$$
    \item Whenever $t_{a},t_{a+1},...,t_b$ is a sequence such that either no $C_{t_i}\subseteq C_0$ or only $C_{t_a}\subseteq C_0,$ then $\prod_{a\leq i<b}y_i\in C_{t_a}.$
\end{enumerate}

Notice that if we do not insist on satisfying property (2) above then we may simply use the van der Waerden theorem and Lemma \ref{lem:ps} to choose the $y_i$ and $t_i$ at each step.  This is precisely how Moreira's proof of Theorem \ref{M} proceeds.  Therefore, if we can carefully choose the numbers $y_i$ to satisfy property (2) then we will be done.  For this, we use Theorem \ref{pos++}.

To begin, let $A_0=C_0$ and $t_0=0$.

Inductively assume we have defined a sequence of  monochromatic piecewise syndetic sets $A_0,...,A_m,$ and sequences $y_0,...,y_{m-1},t_0,...,t_{m}\in \mathbb{N}$ satisfying properties (1) and (2) above.

As we know $t_0=0$, first suppose $t_m=0$.  By Theorem \ref{pos++} we may choose a sequence $S_0,S_1,...\subset \mathbb{N}$ satisfying the conclusion of the theorem, where $f(n)\rightarrow \infty$ and $f(0)$ is large enough so that $A_{m}\cap\bigcap_{f=1}^{\prod_{j<m} k!y_j^3} (A_{m}-fy_m)$ is piecewise syndetic for some $y_m\in FS(S_0)$ (using Lemma \ref{thm:poly}).  Let $A_{m+1}$ and $t_{m+1}$ be defined as in property (1) of the induction.

If $t_{m+1}=0$ then we are in a position to repeat the process just described.  Otherwise, we know $A_{m+1}\subseteq C_1$.  Using Lemma \ref{thm:poly}, we know that for all large enough $n$ there is a $y_{m+1}\in FS(S_n)$ such that $A_{m+1}\cap\bigcap_{f=1}^{\prod_{j<m+1} k!y_j^3} (A_{m+1}-fy_{m+1})$ is piecewise syndetic.  Define $A_{m+2}$ and $t_{m+2}$ as in property (1) of the induction, and repeat this construction for all $A_{m+i}$ (choosing the $y_j$ from different $S_n$) until $t_{m+i}=1.$  Note that the $y_m.y_{m+1},...,y_{m+i}$ chosen this way satisfy property (2) the induction as the $S_i$ satisfy Theorem \ref{pos++}.

With the induction complete, by the pigeonhole principle there are either $n+1$ numbers $i$ such that $t_i=0$ or else there are arbitrarily many consecutive integers $a,...,b$ such that $t_i=1$ for $i\in \{a,...,b\}$.  In either case, let $c_0,..,c_n$ be these integers in increasing order.  Notice that the elements of $\bigcup_i A_{c_i}\subseteq C$ for some color class $C$.  

For integers $i<j$, let $x_{i,j}=y_{i}y_{i+1}...y_{j-1}$.  By property (2) of our construction we know that $x_{c_i,c_{i+1}}\in C$.  Moreover, by property (1) we know that $A_{b}\subseteq x_{a,b}A_{a}$ for each $a<b.$

To complete the proof, let $x\in A_b$ and $f\in [k!\prod y_s ] $.  We will show that $x/x_{a,b}+fx_{a,b}\in A_a$ for $a<b$.  We have: $$x_{a,b}(x/x_{a,b}+fx_{a,b})\in A_b+fx_{a,b}^2\subset y_{b-1}(A_{b-1}-fx_{a,b}^2/y_{b-1})+fx_{a,b}^2 $$ $$\subset y_{b-1}A_{b-1}\subset y_{b-1}x_{a,b-1}A_a= x_{a,b}A_a.$$

As $x/x_{a,b}+fx_{a,b}\in A_a,$ we may repeat the above argument to show that $x/(x_{a',a}\cdot x_{a,b})+f'x_{a',a}+fx_{a,b}/x_{a',a}\in A_{a'}$ for $a'<a,$ and so on, obtaining all of the sums of the desired type.  Applying this fact to $b=c_n$ we see that $x_0=x/x_{c_0,c_n}$ and $x_i=x_{c_{i-1},c_i}$ for $i>0$ are the desired integers.
\end{proof}

\subsection{The thick case}

Applying Theorem \ref{thm:local} reduces the proofs of Theorems \ref{main} and \ref{m2} to the case where at least one of the color classes is multiplicatively thick.  In this situation we can 'almost' ignore color changes caused by multiplication, and so the case analysis needed to complete the proofs is essentially the same as that needed to prove that any $2$-coloring of $\{1,...,5\}$ contains monochrochromatic sets of the form $\{x,y,x+y\}.$  

In the following we fix a $2$-coloring $\mathbb{N}=A\cup B.$

\begin{proof}[Proof of Theorem \ref{m2}]
By Theorem \ref{thm:local} we may assume that there is some minimal left ideal $L\subseteq (\beta\mathbb{N},\cdot)$ such that $L\subseteq \overline{A}.$ By definition this tells us that $A$ is multiplicatively thick, and so for any finite $F\subset \mathbb{N}$ there is an $a\in A$ such that $Fa\subset A.$  Iterating this property, we find a sequence $n<a_0,...,a_4\in A$ such that $[n^2a_0^2 \cdot...\cdot a_i^2] a_{i+1}\subset A$ for all $i\in [4]$.   

From here the proof follows from considering several cases:

\begin{enumerate}
    \item We may assume that $a_j+na_i\in B$ for any $i<j,$ as otherwise $x=a_j$ and $y=a_i$ are as desired.
    
    \item We may assume $(a_3+na_2)(a_1+na_0)\in B,$ as otherwise $x=(a_1+na_0)a_3$ and $y=(a_1+na_0)a_2$ are as desired.
    
    \item We may assume $(a_3+na_2)+n(a_1+na_0)=a_3+n(a_2+a_1+na_0)\in A,$ as otherwise we can take $x=a_3+na_2$ and $y=a_1+na_0.$
    
    \item We may assume $a_4+n((a_3+na_2)+n(a_1+na_0))\in B,$ as otherwise we can take $x=a_4$ and $y=(a_3+na_2)+n(a_1+na_0).$
    
    \vspace{1mm}
    
    From here we finish by considering the color of $y=a_2+a_1+na_0.$
    
    \item If $y\in A,$ then we can let $x=a_3.$  By construction of the $a_i$ we know $x\in A$ and $xy\in A,$ and by point (3) we know that $x+ny\in A.$
    
    \item If $y\in B,$ then consider $x=a_4+na_3.$  By point (1) we know $x\in B$ and by point (4) we know $x+ny=a_4+n(a_3+a_2+a_1+na_0)\in B,$ so we may assume that $xy\in A.$  To finish we can let $x'=a_4y$ and $y'=a_3y.$  Then by the construction of the $a_i$ we know $x',y',x'y'\in A,$ and by the above we know $x'+ny'=xy\in A,$ so these numbers are as desired.
\end{enumerate}

\end{proof}

Theorem \ref{main} will follow from a very similar, albeit more index heavy, analysis.

\begin{proof}[Proof of Theorem \ref{main}]
By Theorem \ref{thm:local} we may assume that there is a minimal left ideal $L\subseteq (\beta\mathbb{N},\cdot)$ such that $L\subseteq \overline{A}$.  By definition this tells us that $A$ is multiplicatively thick, and thus we may find a sequence $(a_i)_{i\in\omega}$ such that $[a_0^n \cdot...\cdot a_i^n] a_{i+1}\subset A$ for all $i\in \mathbb{N}.$  From here we just have an annoying case analysis that is very similar to the analysis that shows that any two coloring of the set $\{1,n,n+1,n^2,n^2+n-1\}$ contains $n$ (not necessarily distinct) $x_i$ such that $\{x_i,\sum x_i\}$ is monochromatic.

\begin{enumerate}
    \item We may assume $\sum_{i\in I}a_i\in B$ for any $I\subseteq \mathbb{N}$ with $|I|=n.$
    
    \item We may assume $\prod_{i\leq j}c_i\in B$ for any sequence of numbers $c_1,...,c_n$ which are the sum of $n$ $a_i$ in increasing order (i.e. the index of all of the elements that sum to $c_i$ is less than the index of any of the elements that sum to $c_j$ for $i<j$).  Otherwise, for $c_j=b_1+...+b_n$ we can take $x_i=\prod_{k<j}c_jb_{n+1-i}$ to finish.
    
    \item We may assume $\sum_{i\in [n]}c_i\in A$ for $c_i$ as above, as otherwise this and Point (2) would ensure that the $c_i$ are the desired elements of $B.$  In particular, we may assume that the sum of any $n^2$ $a_i$ is in $A.$
    
    \item We may assume that the sum of $n^2+n-1$ many $a_i$ is in $B,$ as otherwise letting $x_1$ be the sum of the $n^2$ smallest indexed $a_i$ and $x_{1+i}$ the $n^2+i$'th $a_i$ gives the desired integers by Point (3).
    
    \vspace{1mm}
    
    To finish, consider numbers $b=\sum_{i\in I} a_i$ where $|I|=n+1.$  We case on whether there are infinitely many such numbers in $A$ or in $B.$
    
    \item If there are infinitely many such $b\in A,$ consider index sets $I_2,...,I_{n}$ where $|I_i|=n+1,$ $\max{I_i}<\min{I_j}$ for $i<j,$ and $x_i=\sum_{i\in I_i}a_i\in A.$  Let $x_1=a_k$ where $k>\max{I_{n}}.$ Then $\prod_{i\leq j}x_i\in A$ by the choice of $x_1,$ and $\sum x_i\in A$ by Point (3).
    
    \item If there are infinitely many such $b\in B,$ consider index sets $I_2,...,I_{n}$ where $|I_i|=n+1,$ $\max{I_i}<\min{I_j}$ for $i<j,$ and $x_i=\sum_{i\in I_i}a_i\in B.$  Let $I_1$ be an index set such that $|I_1|=n,$ $\min{I_1}>\max{I_{n}},$ and $x_1=\sum_{i\in I_n}a_i\in B$ by Point (1).  Then $\sum x_i$ is the sum of $(n-1)(n+1)+n=n^2+n-1$ $a_i,$ and so is in $B$ by Point (4).  Therefore, if $x_1\cdot \prod_{2\leq i\leq j}x_i\in B$ for all $j<n$ then we're done, so we may assume that $x_1\cdot \prod_{2\leq i\leq j}x_i\in A$ for at least one $j.$ In this case, suppose that $I_1$ indexes the elements $b_1,...,b_n\in (a_i)_{i\in\mathbb{N}}.$ Let $x'_i=b_i\cdot \prod_{2\leq i\leq j}x_i.$  Then $x'_i\in A$ and $\prod_{i\leq j}x'_i\in A$ for each $j\in [n]$ by the choice of the $(a_i),$ and $\sum x'_i\in A$ by the previous sentence.
    
\end{enumerate}

\end{proof}

\section{Concluding remarks}

While the proof of Theorem \ref{thm:local} can be adapted to deal with some of the cases that arise when considering partitions of $\mathbb{N}$ with more than $2$ colors (for example, when each color class is syndetic), the proof of the thick case seems to be much less versatile.  Finding a more abstract approach to this case seems like a natural first barrier to attempt to overcome when generalizing these results.  In upcoming work with Sabok we make some progress towards this and give a somewhat different approach to the thick case that, in particular, allows us to generalize Theorem \ref{m2} to arbitrarily long arithmetic progressions.  However, this approach still relies on some analysis that does not generalize naturally to more than $2$-colorings.

Finally, we mention that the situation is quite different when considering coloring of $\mathbb{Q}$ rather than $\mathbb{N},$ where the thick case is not nearly as problematic.  In fact, in upcoming work with Sabok we prove (generalizations of) Theorem \ref{M} for any finite coloring of $\mathbb{Q}.$

\end{document}